\documentclass[11pt,a4paper]{amsart}

\usepackage[utf8]{inputenc}

\title[$L^p$-estimates of extensions of holomorphic functions defined on a ] %non-reduced ] %subvariety]
{$L^p$-estimates of extensions of holomorphic functions defined on a non-reduced subvariety}

\author{Mats Andersson}  
\thanks{The author was partially supported by grants from the Swedish Research Council.}
\address{Department of Mathematical Sciences, Division of Algebra and Geometry,
Chalmers University of Technology and the  University of Gothenburg,
SE-412 96 G\"{o}teborg, Sweden}
\email{matsa@chalmers.se} % 

\usepackage{amsmath}
\usepackage{amsthm,amssymb,latexsym}
\usepackage{url}
\usepackage{mathrsfs}
\usepackage{graphicx}
\usepackage{geometry}
\newtheorem{thm}{Theorem}[section]
\newtheorem{lma}[thm]{Lemma}

\newtheorem{prop}[thm]{Proposition}

\theoremstyle{definition}

\theoremstyle{remark}

\newtheorem{preremark}[thm]{Remark}
\newtheorem{preex}[thm]{Example}

\newenvironment{remark}{\begin{preremark}}{\qed\end{preremark}}
\newenvironment{ex}{\begin{preex}}{\qed\end{preex}}

\newcommand{\Ker}{{\text{Ker}\,}}
\newcommand{\Ck}{\mathbb{C}}
\newcommand{\C}{\mathbb{C}}

\newcommand{\dbar}{\bar{\partial}}
\newcommand{\A}{\mathscr{A}}

\newcommand{\J}{\mathcal{J}}
\newcommand{\I}{\mathcal{I}}
\newcommand{\E}{\mathscr{E}}

\newcommand{\Ok}{\mathscr{O}}

 \newcommand{\ra}{{\rangle}}
\newcommand{\la}{{\langle}}

\newcommand{\CH}{\mathcal{{ CH}}}
\newcommand{\B}{\mathbb{B}}

\newcommand{\nbh}{neighborhood }
\newcommand{\1}{{\bf 1}}
\newcommand{\w}{\wedge}

\newcommand{\Homs}{{\mathcal Hom}}
\newcommand{\Hom}{{\text{Hom}\,}}

\newcommand{\N}{{\mathcal N}}

\newcommand{\La}{{\mathcal L}}

\newcommand{\psh}{plurisubharmonic }
\numberwithin{equation}{section}

\usepackage{xcolor}

\date{\today}

\begin{document}

\nocite{*}
\bibliographystyle{plain}

%\begin{document}

%\end{document}

\begin{abstract}
Let $D$ be a strictly  pseudoconvex domain in $\C^N$ and $X$ a pure-dimensional
non-reduced subvariety that behaves well at $\partial D$. We 
provide $L^p$-estimates of extensions of holomorphic functions defined on $X$.
\end{abstract}

\maketitle
\thispagestyle{empty}

\section{Introduction}\label{intro}
Let $D$ be a pseudoconvex domain in $\Ck^N$ and let $X$ be a smooth submanifold
of dimension $n$.  For any holomorphic function $\phi$ on $X$ there is a holomorphic
 extension $\Phi$ to $D$.  The celebrated Ohsawa-Takegoshi theorem, \cite{OT},  provides very precise weighted 
 $L^2$-estimates of such extensions. This theorem, and various variants, have played a decisive role
 in complex analysis and algebraic geometry during the last decades, see, e.g., \cite{Oh}. 
 There are also quite recent extension results, see, e.g.,  \cite{CDM} and \cite{De}, 
 obtained by $L^2$-methods,  in certain cases when $X$ is not reduced. 
 
In case $D$ is strictly pseudoconvex 
there are $L^p$- and $H^p$-estimates of extensions from smooth submanifolds,
based on integral representation, see
 \cite{Amar,Cum, HG}. Notably is that if $D$ is strictly pseudoconvex, and $X$ behaves reasonably at $\partial D$,
 then any bounded holomorphic function on $X$ admits a bounded extension. 
 In \cite{AM} there are estimates of extensions from non-smooth hypersurfaces. 
 These results are based on integral formulas for representing the extensions
 or for solving  $\dbar$-equations in $D$.

\smallskip 

Let $i\colon X\to D$ be a 
non-reduced subspace of pure dimension $n$ of a pseudoconvex domain $D$.
That is, we have a coherent ideal sheaf $\J\to D$ of pure dimension $n$ so that
the sheaf $\Ok_X$ of holomorphic functions on $X$, the structure sheaf,
is isomorphic to $\Ok_D /\J$. 
We thus have a natural mapping $i^*\colon \Ok_D \to \Ok_X$,
and we say that $\Phi$ is an extension of a function $\phi$ on $X$,
or that $\Phi$ interpolates $\phi$, if $i^*\Phi=\phi$. 
In \cite{Anorm} we introduced a pointwise coordinate invariant norm $|\phi|_X$ of holomorphic functions $\phi$ 
on $X$.   In this paper we will only consider $X$ such that the underlying
reduced space $Z$, i.e., the zero set of $\J$,  is smooth. In this case 
the norm  $|\phi|_X$ is well-defined on compact subsets up to multiplicative constants.
Recall that a holomorphic differential operator $L$ in $D$ is Noetherian with respect to $\J$ if
$L\Psi$ vanishes on $Z$ as soon as $\Psi$ is in $\J$.  
Such an $L$ induces a mapping $\La\colon\Ok_X\to\Ok_Z$ that we also
call a Noetherian operator. % 
In \cite{Anorm} we introduced a locally finitely generated coordinate invariant $\Ok_D$-sheaf
$\N_X$ of Noetherian operators such that $\phi=0$ if and only if $\La\phi=0$
for all $\La$ in $\N_X$. 
We defined the pointwise norm locally as
\begin{equation}\label{pork}
|\phi(z)|_X=\sum_j |\La_j\phi(z)|,
\end{equation}
where $\La_\ell$ is finite set of generators of $\N_X$. For a precise description of $\N_X$,
see Section~\ref{prel}.  Notice that  $|\phi(z)|_X=0$ in an open set if and only if
$\phi$ vanishes identically there. Roughly speaking $|\cdot|_X$ is the smallest invariant norm 
with this property, see Remark~\ref{skola}.

By means of $|\cdot|_X$ we can define $L^p$-norms of $\phi$ in $\Ok_X$.  
It is then natural to look for $L^p$-estimates
of extensions of holomorphic functions on $X$. In this paper we 
present a couple of such results when $D$ is strictly pseudoconvex.  We do not look for  
the most general possible statements but our aim is to point out some new ideas. In order not to
conceal them by technicalities we assume that $X$ behaves well at the boundary of $D$.
Here is our main result.

\begin{thm}\label{thmA}
Let $D\subset\subset \Omega\subset\Ck^N$ be a strictly pseudoconvex domain with smooth boundary, and 
let $i\colon X\to \Omega\subset\C^N$ be a non-reduced subspace of pure dimension $n$ such that
$Z=X_{red}$ is smooth and intersects $\partial D$ transversally.  
Assume that $\Ok_X$ is Cohen-Macaulay
at each point on $Z\cap\partial D$. Let $\kappa=N-n$. Assume that $1\le p<\infty$ and that $r > -1$. 
Let $\delta(z)=dist(z,\partial D)$ be the distance to the boundary.
Each holomorphic function $\phi$ in $\Ok(X\cap D)$ admits a holomorphic extension $\Phi\in\Ok(D)$
such that 
 \begin{equation}\label{trauma}
\int_D \delta^{r} | \Phi|^p dV_D \le C_{r,p}^p \int_{Z\cap D} \delta ^{\kappa+r} |\phi|^p_X dV_Z,
\end{equation}
provided that the right hand side is finite. 
\end{thm}
 
 Here $dV_D$ and $dV_Z$ denote some volume forms on $D$ and $Z$, respectively.
 Since $X$ is defined in $\Omega$,  the $L^p$-norms are well-defined up to multiplicative constants.

\smallskip
The transversally condition means that if $\rho$ is a defining function for $D$ and $(\zeta,\eta)$ are local
coordinates such that $Z=\{\eta=0 \}$,
then $\partial\rho\w d\eta_1\w\ldots \w d\eta_\kappa$ is non-vanishing
on $\partial D\cap Z$.  In particular,  $D\cap Z$ is a strictly pseudoconvex domain in $Z$ with smooth
boundary.

\smallskip
Assume that $D\subset \C_{\zeta,\tau}^{n+\kappa}$ is the unit ball, $Z=\{\tau=0\}$  and  $X=Z$ is reduced.
If $\phi(\zeta)$ is holomorphic on $Z\cap D$ and $\Phi(\zeta,\eta)=\phi(\zeta)$ is the trivial extension to 
the entire ball, and $\delta(\zeta,\tau)=1-|\zeta|^2-|\tau|^2$, then  
$$
\int_D \delta^r |\Phi|^pdV_D=c_{r,\kappa}\int_{Z\cap D}\delta^{r+\kappa}|\phi|^pdV_Z,
$$
where
$c_{r,\kappa}=\pi^\kappa/(r+1)\cdots (r+\kappa)$.
It follows that the estimate \eqref{trauma} is sharp up to the constant $C_{r,p}$ when $X$ is reduced.  
In the non-reduced case it is not, as we will see in our second result.

Assume that  $Z$ is a smooth hypersurface in $\Omega$ defined by the function $f$ in $\Omega$,
i.e., $Z=Z(f)$ and $df\neq 0$ on $Z$, 
let $\J=\la f^{M+1}\ra$ and let $\Ok_X=\Ok_\Omega/\J$.    It turns out that then $\N_X$ is generated
by all differential operators of order at most $M$, so that  
 $$
 |\phi(z)|_X=\sum_{k=0}^M \sum_{|\beta|=k}\big|\partial^\beta\phi|_z\big|, \quad z\in Z.
$$
%where $\partial^\beta\phi$ means all derivatives

\begin{thm}\label{thmB}
Let $D\subset\subset \Omega$ be as in Theorem~\ref{thmA}. Assume that $Z$ is a 
smooth hypersurface in $\Omega$ that intersect $\partial D$ transversally. Assume that
that $Z$ is defined by the function $f$ and let $\Ok_X=\Ok_D/\la f^{M+1}\ra$.
Morever, assume that $1\le p<\infty$,  $r>-1$, and let $\delta$ be the distance to the boundary.
Each function $\phi$ on $\Ok(D\cap X)$ has an extension $\Phi\in \Ok(D)$ such that 
$$
\int_D \delta^r |\Phi|^pdV_D\le C_{r,p}^p \sum_{k=0}^M\int_{Z\cap D} \delta^{r+1+k/2}
\sum_{|\beta|=k} |\partial^\beta \phi|^pdV_Z,
$$
provided that the right hand side is finite.
\end{thm}

Thus the requirement is less restrictive for higher derivatives of $\phi$.

\smallskip

The extension $\Phi$ in Theorems~\ref{thmA} and \ref{thmB} 
is obtained by an integral formula, that in turn is constructed by means
of the residue currents in \cite{AW1} and the division-interpolation formulas in
\cite{Aint2}.  A main novelty is the technique to carry out the estimates 
in terms of the norm in \cite{Anorm}.    

When $D$ is a ball the extension formula is explicitly given in terms of the
residue current associated with $X$.  In the general case the analogously constructed
formula does not provide a holomorphic extension, so it has to be slightly modified by a technique
inspired by  a classical idea of Kerzman-Stein and Ligocka, see, e.g.,  \cite{Range}. 
To this end we have to construct a 
linear solution operator for the $\dbar$-equation for $\dbar$-closed
smooth $(0,1)$-forms in $\E\J$ for a quite arbitrary ideal sheaf $\J$, Theorem~\ref{CC}.

In Section~\ref{prel} we recall the definition of the norm $|\cdot|_X$, and in Section~\ref{exsec}
we give some examples of computations of the norm $|\cdot |_X$  and applications
of Theorem~\ref{thmA}.
In Sections~\ref{kodiak} 
to \ref{pokemon} we recall  
the residue currents associated with $X$, and we make the construction of
interpolation-division formulas in strictly pseudoconvex domains.
The remaining sections are devoted to the proofs of  Theorem~\ref{thmA},
Theorem~\ref{thmB}, and Theorem~\ref{CC}.

\smallskip
\noindent {\bf Acknowledgement}
The author would like to thank H\aa kan Samuelsson Kalm for valuable discussion
on various questions in this paper, and the referee for careful reading and important comments.

\section{The pointwise norm on $X$}\label{prel}
Let $\Omega\subset\C^N$ be an open pseudoconvex domain, let $Z$ be a submanifold of dimension $n<N$, and
let $\kappa=N-n$.  
The $\Ok_\Omega$-sheaf of Coleff-Herrera currents, $\CH_\Omega^Z$, introduced by  Bj\"ork, see \cite{BjAbel},
is the set of  $\dbar$-closed $(N,\kappa)$-currents $\mu$ in $\Omega$ with support on $Z$ 
such that $\bar h\mu=0$ for all holomorphic  functions $h$ that vanishes on $Z$. 
It is well-known that  $\CH_\Omega^Z$ is coherent.  
 Notice that if $\J\subset\Ok_\Omega$ is 
 an ideal sheaf with zero set $Z$, then  
$\Homs(\Ok_\Omega/\J,\CH_\Omega^Z)$ is the subsheaf of
$\mu$ in  $\CH_\Omega^Z$ that are annihilated by $\J$.  

\begin{remark} If $Z$ is not smooth, then $\CH_\Omega^Z$ is  defined in the same way, but one must impose
an additional regularity condition at $Z_{sing}$,  see, \cite{BjAbel} or, e.g., \cite[Section~2.1]{AL}.
\end{remark}

Consider the embedding $i\colon X\to \Omega\subset\C^N$. 
Locally, in say $U\subset\Omega$, we have coordinates
$(\zeta,\tau)=(\zeta_1,\ldots, \zeta_n, \tau_1,\ldots, \tau_\kappa)$ 
so that $Z\cap U=\{ \tau=0\}$. Then the mapping $\pi\colon U\to Z\cap U$,
$(\zeta,\tau)\mapsto \zeta$ is a submersion, and locally any submersion appears in this way.
%and $\pi$ is the projection $(\zeta,\tau)\mapsto \zeta$. 
%
If $\mu$ is a section of $\Homs(\Ok_\Omega/\J,\CH_\Omega^Z)$ in $U$ and $\phi$ is a holomorphic
function, then $\pi_*(\phi\mu)$ is a holomorphic $(0,n)$-form on $Z\cap U$ that only depends
on the image $i^*\phi$ of $\phi$ in $\Ok(X\cap U)$.   If  
$d\zeta=d\zeta_1\w\ldots\w d\zeta_n$ thus
\begin{equation}\label{tomat}
\pi_*(\phi\mu)= \La\phi \, d\zeta
\end{equation}
defines a holomorphic differential operator, a Noetherian operator,  $\La\colon \Ok(X\cap U)\to \Ok(Z\cap U)$,
cf.~Section~\ref{intro}.  
Following \cite{Anorm} we define $\N_X$ as the set all such local operators $\La$ obtained from 
some  $\mu$ in 
$\Homs(\Ok_\Omega/\J,\CH_\Omega^Z)$ and local submersion.
It follows from \eqref{tomat} that if $\xi$ is in $\Ok_Z$, then $\xi \La\phi=\La(\pi^*\xi \phi)$.
Thus $\N_X$ is a left $\Ok_Z$-module. It is in fact coherent, in particular it is locally finitely generated.   
Recall that our norm $|\cdot|_X$ is defined by \eqref{pork}.

 \section{Examples} \label{exsec}
We now give some examples 
how to compute the Noetherian operators $\La$, the norm $|\cdot |_X$, and
where Theorem~\ref{thmA} is applicable.
We keep the notation from the previous section.

\begin{ex}\label{simple1}
Assume that we have an embedding and local coordinates $(\zeta,\tau)$ as above in $U\subset \Omega$.
Let  $M=(M_1,\ldots,M_\kappa)$ be a tuple of non-negative integers and consider the ideal 
sheaf
$$
\I=\left< \tau_1^{M_1+1},\ldots, \tau_\kappa^{M_{\kappa}+1}\right>.
$$
Let $\hat X$ be the analytic space with structure sheaf $\Ok_{\hat X}=\Ok_\Omega/ \I$.
For a multiindex $m=(m_1,\ldots, m_\kappa)$, 
%we use the short-hand notation
%\begin{equation} \label{eq:shorthand}
%    \dbar\frac{d\tau}{\tau^m} = \dbar\frac{d\tau_1}{\tau_1^{m_1}}\w\ldots\w\dbar\frac{d\tau_\kappa}{\tau_\kappa^{m_\kappa}}
%\end{equation}
$m\le M$ means that $m_j\le M_j$ for $j=1,\ldots, \kappa$.
Any $\psi$ in $\Ok_{\hat X}$ has a unique representative in $\Omega$ of the form
\begin{equation}\label{simple0}
\psi=\sum_{m\le M} \hat\psi_m(\zeta) \otimes\tau^m.
\end{equation}
The tensor product of currents
\begin{equation}\label{hatmu}
\hat\mu=\dbar\frac{d\tau_1}{\tau_1^{M_1+1}}\w\ldots\w\dbar\frac{d\tau_\kappa}{\tau_\kappa^{M_\kappa+1}} \w d\zeta
=:\dbar\frac{d\tau}{\tau^{M+\1}}\w d\zeta,
\end{equation}
where $d\tau_j/\tau_j^{M_j+1}$ are principal value currents, 
is a Coleff-Herrera current in $U$.  
If $\varphi = \varphi_0(\zeta,\tau) d\bar{\zeta}$ is a test form, then
\begin{equation}\label{hatmuintegral}
\hat\mu.\varphi=\dbar\frac{d\tau}{\tau^{M+\1}} \w d\zeta.\varphi
=  \frac{(2\pi i)^\kappa}{M!} \int_{\zeta} \frac{\partial^{|M|} \varphi_0}{\partial \tau^M} (\zeta,0) d\zeta \wedge d\bar{\zeta},
\end{equation}
where $M!=M_1!\cdots M_\kappa!$.
In particular, $\tau_j^{M_j+1}\hat\mu=0$ for each $j$ and thus $\hat\mu$ is in
the $\Ok_\Omega$-module (and $\Ok_{\hat X}$-module)
$\Homs(\Ok/\I,\CH_\Omega^Z)$.  It is in fact a generator for this module, see, e.g.,
\cite[Theorem~4.1]{ACH}.
If $\pi$ is the simple projection $(\zeta,\tau)\mapsto \zeta$ and $\psi$ is holomorphic, then
it follows from \eqref{hatmuintegral} that 
$$
%\La \psi d\zeta=
\pi_*(\psi\hat\mu) =  \frac{(2\pi i)^\kappa}{M!} \frac{\partial^{|M|} \psi}{\partial \tau^M} (\zeta,0) d\zeta
$$
and hence, cf.~\eqref{tomat}, 
$$
\La\psi= \frac{(2\pi i)^\kappa}{M!} \frac{\partial^{|M|} (\psi\gamma)}{\partial \tau^M} (\zeta,0).
$$
For a general projection $\pi$, its associated Noetherian operator $\La$  will involve derivatives with respect to $\zeta$
as well, cf.~\cite[Section~2]{Anorm}.  

One can check that the set  $\mu_\alpha:=\tau^\alpha\hat\mu$, $\alpha\le M$, generates 
the $\Ok_Z$-module $\N_{\hat X}$ in $U$. 
If $\Psi(\zeta,\tau)$ is any representative in $U$ for $\psi$ in $\Ok_{X'}$, then it follows from
\cite[Proposition~4.6]{Anorm}, with $a_k=1$, cf.~\cite[(4.22)]{Anorm}, that
\begin{equation}\label{tomat2}
|\psi|_{\hat X}\sim \sum_{m\le M, \ |\alpha|\le |M-m|}
\Big|\frac{\partial}{\partial \tau^m \partial \zeta^\alpha}\Psi(\zeta,0)\Big|.
%  
%\sum_{m\le M, \ |\alpha|\le |M-m|} \Big|\frac{\partial\hat{\psi}_m}{\partial \zeta^\alpha}\Big|.
\end{equation}
\end{ex}

Possibly after shrinking $\Omega$ somewhat, there is a finite number 
$\mu_1,\ldots, \mu_\rho$ of sections that generate  the $\Ok_\Omega$-module $\Homs(\Ok_\Omega/\J,\CH_\Omega^Z)$ 
in $\Omega$. 
Locally, in $U\subset \Omega$, we can choose $M$ in Example~\ref{simple1} so that 
$\I\subset \J$ in $U$.  
Since $\hat\mu$ generates $\Homs(\Ok_\Omega/\I,\CH_\Omega^Z)\supset \Homs(\Ok_\Omega/\J,\CH_\Omega^Z)$,
there are holomorphic functions 
$\gamma_1, \ldots, \gamma_\rho$  
such that
\begin{equation}\label{tomat3}
\mu_j=\gamma_j \hat\mu, \quad j=1,\ldots,\rho
\end{equation}
in $U$.
It follows from \cite[(4.22)]{Anorm} that 
\begin{equation}\label{tomat4}
|\phi|_X\sim \sum_{j=1}^\rho  |\gamma_j \phi|_{\hat X}.
\end{equation}

\begin{ex}
Let $f_1,\ldots,f_\kappa$ be holomorphic functions in a \nbh $\Omega$ of the (closure of the) unit ball $D$ 
in $\Ck^N$ such that 
$df_1\w \ldots \w df_\kappa\neq 0$ on their common zero set $Z$. Moreover,
assume that $Z$ intersect $\partial D$
transversally. Let $\J=\la f_1^{M_1+1}, \ldots, f_\kappa^{M_\kappa+1}\ra$ and let $X$ be the non-reduced space
associated with $\Ok_\Omega/\J$.  Then Theorem~\ref{thmA} applies to $X$ and $D$. 

At a given point  $p\in Z\cap\partial D$  we can assume, possibly after reordering, that  the coordinates in $\Ck^N$ 
are $(z,w)=(z_1,\ldots, z_n, w_1,\ldots,w_\kappa)$ and that 
\begin{equation}\label{pojk}
df_1\w \ldots \w df_\kappa\w dz_1\w \ldots \w dz_n\neq 0
\end{equation}
 at $p$. In a \nbh $U$ of   $p$,
therefore $\zeta=z, \ \tau=f$ are local coordinates.    The assumption \eqref{pojk} means that
the matrix $B=\partial f/\partial w$ is invertible, and we have
\begin{equation}\label{plupp}
\frac{\partial}{\partial\tau}= B^{-1}\frac{\partial}{\partial w},
\quad 
\frac{\partial}{\partial\zeta}=\frac{\partial}{\partial z}-\frac{\partial f}{\partial z}B^{-1} \frac{\partial}{\partial w}.
\end{equation}
If $\phi$ is a function on $X$, then $|\phi|_X$ is given by \eqref{tomat2},  
which can be expressed in terms of the original coordinates $(z,w)$ by \eqref{plupp}.
\end{ex}

\begin{remark}[The pointwise norm when $X$ is Cohen-Macaulay]\label{skola}
Assume that we have a local coordinates $(\zeta,\tau)$ in $U\subset \Omega$ as above.
Assume furthermore that $\Ok_X$ is Cohen-Macaulay.
Then one can find monomials $1, \tau^{\alpha_1}, \ldots,\tau^{\alpha_{\nu-1}}$
such that each $\phi$ in $\Ok_X$ has a  unique representative
\begin{equation}\label{skata1}
    \hat \phi=\hat{\phi}_0(\zeta)\otimes 1 +\cdots +\hat{\phi}_{\nu-1}(\zeta)\otimes \tau^{\alpha_{\nu-1}},
\end{equation}
where $\hat{\phi}_j$  are in $\Ok_Z$, see, e.g., \cite[Corollary~3.3]{AL}, cf.~\eqref{simple0}.
In this way $\Ok_X$
becomes  a free $\Ok_Z$-module.  By \cite[Theorem~4.1 (iii)]{Anorm}
 $|\cdot|_X$ is the smallest norm such that
\begin{equation}\label{tomat5}
\sum|\hat\phi_j(\zeta)|\le C |\phi(\zeta)|_X,
\end{equation}
%and \eqref{tomat5} holds 
for any choice of local coordinates and monomial basis. 
\end{remark}

We now consider a space $X$ with a non-Cohen-Macaulay point,
see~\cite[Section~5]{Anorm}.

\begin{ex}\label{noncm}
Consider the $2$-plane $Z=\{w_1=w_2=0\}$ in $\Omega=\C^4_{z_1,z_2,w_1,w_2}$,
let
$$
\J=\langle w_1^2,w_2^2, w_1w_2, w_1z_2-w_2z_1\rangle.
$$
%and let $X$ be the associated non-reduced space with
Then the associated non-reduced space $X$  has pure dimension $2$ and is Cohen-Macaulay except at the point $0$,
see, \cite[Example~6.9]{AL}. It is also shown there that $\Homs_\Ok(\Ok/\J,\CH^Z_\Omega)$ is generated by
$$
\mu_1=\dbar\frac{dw_1}{w_1}\w\dbar\frac{dw_2}{w_2}\w dz_1\w dz_2,\quad
\mu_2=(z_1w_2+z_2w_1)\dbar\frac{dw_1}{w_1^2}\w\dbar\frac{dw_2}{w_2^2}\w dz_1\w dz_2.
$$
It turns out, see \cite[Section~5]{Anorm},
that the left $\Ok_Z$-module $\N_X$ is generated by %the differential operators
\begin{equation}\label{struma}
1, \ z_2\frac{\partial}{\partial z_1}, \ z_1\frac{\partial}{\partial z_1},\
z_2\frac{\partial}{\partial z_2}, \ z_1\frac{\partial}{\partial z_2}, \
z_1\frac{\partial}{\partial w_1}+z_2\frac{\partial}{\partial w_2}.
\end{equation}
Thus we get, cf.~\eqref{pork},  
\begin{equation}\label{strumpa}
|\phi|_X^2 = |\phi|^2 +|z|^2\big|\frac{\partial\phi}{\partial z_1}\big|^2+|z|^2\big|\frac{\partial\phi}{\partial z_2}\big|^2
+\big|z_1\frac{\partial\phi}{\partial w_1}+z_2\frac{\partial\phi}{\partial w_2}\big|^2.
\end{equation}
 \end{ex}

\begin{ex}
Let $X$ be the space in Example~\ref{noncm} and let $D$ be the unit ball in $\Ck^4$.  Each holomorphic function in
$\phi$ in $\Ok(X\cap D)$ can be written 
\begin{equation}\label{strumpa1}
\phi=\phi_0(z)\otimes 1+ \phi_1(z)\otimes w_1+\phi_2(z)\otimes w_2,
\end{equation}
where $\phi_j$ are holomorphic functions in the unit ball $\B\subset \Ck^2_z$.  One can check that 
\begin{equation}\label{strumpa3}
 h(z)=z_1\phi_1(z)+z_2\phi_2(z).
\end{equation}
is independent of the representation \eqref{strumpa1}. In this way we get a one-to-one correspondence between
functions in $\Ok(X\cap D)$ and pairs $\phi_0, h$ of holomorphic functions in $\B$ such that $h(0)=0$.
Notice that the
rightmost term in \eqref{strumpa} is precisely $|h|^2$.
If now $\phi$ is in $L^p(\B,\delta^r)$, then 
$h$ is in $L^p(\B,\delta^r)$, and $h(0)=0$. It is well-known that one can find 
$\phi_1,\phi_2$ in $L^p(\B,\delta^r)$ such that \eqref{strumpa3} holds.
With such choices, \eqref{strumpa1} defines an extension $\Phi$ of $\phi$ to the entire ball $D$
that satisfies the estimate 
\eqref{trauma} in Theorem~\ref{thmA}, cf.~the discussion after that theorem.

If we deform the embedding of $X$ in $D$ slightly, or replace the unit ball by
a more general strictly pseudoconvex set $D$,  so that the non-Cohen-Macaulay point is still in the interior
of $D$, then Theorem~\ref{thmA} provides a non-trivial extension.  
If this point lies on $\partial D$, then Theorem~\ref{thmA} is not applicable.
\end{ex}

\section{Residue currents associated with a free resolution}\label{kodiak}
If $\J$ is a coherent ideal sheaf in $\Omega$, then we can find a free resolution 
\begin{equation}\label{free}
0\to \Ok(E_{\nu})\stackrel{f_{\nu}}{\to} \Ok(E_{\nu-1})\cdots \stackrel{f_1}{\to} \Ok\to \Ok/\J\to 0
\end{equation}
of $\Ok/\J$ in a slightly smaller pseudoconvex domain that we for simplicity denote by $\Omega$ as well. 
If the (trivial) vector bundles $E_k$ are equipped with Hermitian
metrics we say that \eqref{free} is a hermitian resolution.  For each Hermitian resolution
there are,  \cite{AW1}, associated residue currents 
$$
R=\sum_{k=\kappa}^\nu R_k, \quad   U=\sum_{\ell,k} U_k^\ell,
$$
where $R_k$ are  currents of bidegree $(0,k)$ with support on $Z:=Z(\J)$ that take values in $\Hom(E_0, E_k)\simeq E_k$,
and $U_k^\ell$ are $(0,k-\ell)$-currents that are smooth outside $Z$
and take values in $\Hom(E_\ell,E_k)$. 

\begin{remark}\label{kompass}
The currents $R$ and $U$ are defined even if \eqref{free} is just a pointwise generically exact complex.
In general then $R$ has components $R_k^\ell$ with values in $\Hom(E_\ell,E_k)$ even for $\ell\ge 1$.
\end{remark}

If $\J$ is Cohen-Macaulay, then
one can choose \eqref{free} so that $\nu=\kappa$. In that case the components of $R=R_{\kappa}$
are in $\Homs(\Ok_\Omega/\J,\CH_\Omega^Z)$.   If we only assume that $\Ok_X$ has pure
dimension, then we may have components $R_k$ for $k\le N-1$, see, e.g., \cite{AS,AL}.  They 
can be written,  \cite[Lemma~6.2]{AL}, 
\begin{equation}\label{pjosk1}
R_k\w d\zeta_1\w\ldots\w d\zeta_{N}=a_k\mu, 
\end{equation}
where $\mu$ is in $\Homs(\Ok_\Omega/\J,\CH_\Omega^Z)$
with values in a trivial bundle $F$ and $a_k$ are currents in $\Omega$ that take values in $\Hom(F,E_k)$.
Moreover, $a_k$ are smooth outside the Zariski closed set $W\subset \Omega$ of non-Cohen-Macaulay points on $Z$,
which has positive codimension on $Z$. 
The currents $a_k$ are {\it almost semi-meromorphic}  in the terminology from
\cite{AS, AW3}.  For us the important point is that  
\begin{equation}\label{pjosk2}
a_k\mu =\lim_{\epsilon\to 0}\chi(|f|^2/\epsilon) a_k \mu, 
\end{equation}
if $f$ is a holomorphic tuple with zero set $W$ and $\chi$ is a smooth function on $[0,1)$ that is 
$1$ for $t<1/2$ and $0$ for $t>1$.  Notice that for each $\epsilon>0$, 
$\chi(|f|^2/\epsilon) a_k \mu$ is the product of a current and a smooth form.

\smallskip
We have the following duality results.

\begin{prop}\label{stekpanna}
If $\Phi$ is holomorphic, then it is in the ideal $\J$ if and only if
$R\Phi=0$.
If $\Ok_X$ has pure dimension, then $\Phi$ is in $\J$ if and only if
$\mu\Phi=0$ for all $\mu$ in $\Homs(\Ok_\Omega/\J,\CH_\Omega^Z)$.
\end{prop}

The first statement is a basic result in  \cite{AW1}, and the second one is 
proved in \cite{Aext}.

 \section{Integral representation of holomorphic functions}\label{kodiak2}
Following \cite{Aint} we recall a formalism to generate representation formulas for holomorphic functions.
 Let $z$ be a fixed point in $\Omega$, let $\delta_{\zeta-z}$ be contraction
with the vector field
$$
2\pi i\sum_{j=1}^N (\zeta_j-z_j)\frac{\partial}{\partial \zeta_j}
$$
in $\Omega$ and let $\nabla_{\zeta-z}=\delta_{\zeta-z}-\dbar$,
where $\dbar$ only acts on $\zeta$.
We say that a current $g=g_{0,0}+\cdots +g_{n,n}$, where  
$g_{k,k}$ has bidegree $(k,k)$, is a weight with respect to $z$, if 
$\nabla_{\zeta-z} g=0$, $g$ is smooth in a \nbh of $z$, and $g_{0,0}$ is $1$ 
when $\zeta=z$.  Notice that if $g$ and $g'$ are weights, one of which is smooth, then $g'\w g$ is again a weight.
The basic observation is that if $g$ is a weight (with respect to $z$) with compact support in $\Omega$, then
\begin{equation}\label{bas}
\phi(z)=\int_{\zeta\in\Omega} g\phi
\end{equation}
if $\phi$ is holomorphic in $\Omega$, \cite[Proposition~3.1]{Aint}.

If $\Omega$ is pseudoconvex and $D\subset\subset\Omega$, then, see
\cite[Example~1]{Aint2},
we can find  a weight $g$, with respect to $z\in\overline D$,  with compact support in $\Omega$,
such that $g$ depends holomorphically on $z\in\overline D$.
If $D$ and $\Omega$ are balls with center at $0\in\Omega$, then we can take
$$
g=\chi - \dbar\chi \w \frac{\sigma}{\nabla_{\zeta-z}\sigma}
=\chi - \dbar \chi \wedge \sum_{\ell=1}^N \frac{1}{(2\pi i)^\ell}\frac{\zeta \cdot d\bar\zeta\w (d\zeta\cdot d\bar \zeta)^{\ell-1}}
{(|\zeta|^2-\bar\zeta\cdot z)^\ell},
$$
where $\chi$ that is $1$ in a \nbh of $\overline D$, with compact support in $\Omega$, and 
$$
\sigma= \frac{1}{2\pi i}\frac{\zeta \cdot d\bar\zeta}{|\zeta|^2-\bar\zeta\cdot z}.
$$

\subsection{Division-interpolation formulas}\label{poker}
Let $(E,f)$ be a Hermitian resolution of $\Ok/\J$  in $\Omega$ as in Section~\ref{kodiak}. 
In order to construct 
division-interpolation formulas with respect to $(E,f)$, in \cite{Aint2} was introduced the notion
of an associated family $H=(H^\ell_k)$ of Hefer morphisms. The $H^\ell_k$ are holomorphic
$(k-\ell)$-forms with values in $\Hom(E_{\zeta,k}, E_{z,\ell})$ that are connected in the 
following way:  To begin with, $H^\ell_k=0$ if $k-\ell< 0$, and
$H_\ell^\ell$ is equal to $I_{E_\ell}$ when $\zeta=z$. In general,
\begin{equation}\label{heferlikhet}
\delta_{\zeta-z}H^\ell_{k+1}=H_{k}^{\ell} f_{k+1}(\zeta) - f_{\ell+1}(z)H_{k+1}^{\ell+1}.
\end{equation}
If $R$ and $U$ are the associated currents in Section~\ref{kodiak}, then 
$$
HR=\sum H^0_kR_k, \quad H^1U=\sum_k H^1_kU^1_k,
$$
are scalar-valued currents, cf.~Remark~\ref{kompass}. 
It turns out, see \cite[Eq.~(5.4)]{Aint2},  that 
$$
g'=f_1(z)H^1U+HR
$$
is a weight with respect to $z$ for each $z\in\Omega\setminus Z$.
If $g$ is a smooth weight with respect to $z\in \overline D\subset \Omega$, depending
holomorphically on $z$, 
with compact support in $\Omega$
and $\Psi$ is holomorphic in $\Omega$, then $g'\w g$ is a weight with compact support 
with respect to $z\in D\setminus Z$.
By \eqref{bas} we therefore have 
\begin{equation}\label{basic2}
\Psi(z)=\int_{\zeta\in\Omega}  g'\w g\Psi = f_1(z)\int_{\zeta\in\Omega}  H^1U\w g\Psi+
\int_{\zeta\in\Omega} HR\w g \Psi 
\end{equation}
for $z\in D\setminus Z$.
Since the right hand side has a holomorphic extension across $Z$, actually  \eqref{basic2} holds for all $z$ in $D$
by continuity.

Now assume that $\phi$ is a section of $\Ok/\J$ in $\Omega$. Since $\Omega$ is 
 pseudoconvex there is some holomorphic extension $\Psi$ of $\phi$ to $\Omega$. Since $\J$ annihilates
 $R$,  see Section™\ref{kodiak}, the current $R\phi:=R\Psi$ is independent of the extension and thus intrinsic.
Since $f_1(z)$ is in $\J$,  we conclude from \eqref{basic2} that 
\begin{equation}\label{krokant}
\Phi(z)=\int_{\zeta\in\Omega} HR\w g \phi
\end{equation}
is a holomorphic function in $D$ that extends $\phi$.  %such that $\Phi-\phi$ is in $\J$.  
In order to obtain interesting estimates however,
we must replace $g$ by a weight with support on $\overline D$.  

For future reference notice that if $g$ only depends smoothly on $z\in D$, then \eqref{krokant} is a smooth
function in $D$ such that $\Phi-\phi$ is in $\E\J$, where $\E$ is the sheaf of smooth functions.

\section{Integral formulas in strictly pseudoconvex domains}\label{pokemon}
The material in this section is basically well-known but we need it for the construction of our formula. 
Assume that $D\subset\subset\Omega\subset\C^N$ is strictly pseudoconvex with smooth boundary.  
We can assume that $D=\{\rho<0\}$, where  $\rho$ is strictly plurisubharmonic in $\Omega$. 
If $D$ is the ball we can take $\rho=|\zeta|^2-1$.  
%\smallskip
If $D$ is strictly convex, then  
$\delta_{\zeta-z}\partial\rho$ is holomorphic in $z\in D$, and if $\rho$ is strictly convex, then
$$
2{\rm Re\,} \delta_{\zeta-z}\partial\rho\ge \rho(\zeta)-\rho(z)+c |\zeta-z|^2
$$
for some constant $c>0$. 
If
\begin{equation}\label{mos1}
v(\zeta,z):= \delta_{\zeta-z}\partial\rho -\rho(\zeta) =-\rho(\zeta)-\sum_j\frac{\partial\rho}{\partial\zeta_j}(\zeta)(z_j-\zeta_j),
\end{equation}
 because of the strict convexity, therefore
\begin{equation}\label{mos2}
2{\rm Re\, } v(\zeta,z) \ge -\rho(z)-\rho(\zeta)+c|\zeta-z|^2,
\end{equation}
and moreover,  
\begin{equation}\label{mos3}
d({\rm Im\, } v)|_{\zeta=z} =d^c\rho(z)/4\pi.
\end{equation}
Altogether it follows that if $z$ (or $\zeta$)
is a fixed point $p$ on $\partial D$, then  
the level sets of
$|v(\zeta,z)|$ are non-isotropic so-called Koranyi balls around $p$. 
More precisely, if $x_1=-\rho(\zeta)$, $x_2={\rm Im\,} v(\zeta,z)$,
and $x_3, \ldots, x_{2N}$ are chosen so that $x_1,\ldots,x_{2N}$ is a local (non-holomorphic)
coordinate system at $p$ with $x(p)=0$, and $y$ are the corresponding coordinates for $z$, then 
\begin{equation}\label{pontus}
|v(\zeta,z)|\sim x_1 + y_1 + |x_2-y_2|+ \sum_{j=3}^{2N} (x_j-y_j)^2+\Ok(|x-y|^3).
\end{equation}
%Notice that if we take $q=\partial\rho(\zeta)$, then
One can make a similar construction of $v$ 
if $D$ is strictly pseudoconvex. %see Remark~\ref{laban}.
If $D$ is the ball and $\rho=|\zeta|^2-1$, then 
$$
v(\zeta,z)=1-\bar\zeta\cdot z
$$
which is anti-holomorphic in $\zeta$. 
In general, unfortunately, 
$\partial_\zeta v$ will only vanish to first order on the diagonal.  
We need such a function $v$ that is (essentially) anti-holomorphic
in $\zeta$ so we must elaborate the construction.

\subsection{Definition of $v$ in the general case}
First assume that $\rho(z)$ is strictly \psh and real-analytic.  Then close to the diagonal we choose
$v(\zeta,z)$ so that $v(\bar\zeta,z)$ is the (unique) holomorphic extension
of $-\rho(z)$ from the totally real
subspace $\{\zeta=\bar z\}$ of $\{(\bar\zeta,z);\ (\zeta,z)\in \Omega_\zeta\times\Omega_z\}$.
Then 
$\overline{v(z,\zeta)}=v(\zeta,z)$
and $v$ is anti-holomorphic in $\zeta$. 
We can represent $v$ by the power series
\begin{equation}\label{poker1}
v(\zeta,z)=-\sum_\alpha \frac{1}{\alpha!}\frac{\partial^\alpha\rho}
{\partial \zeta^\alpha}(\zeta)(z-\zeta)^\alpha.
\end{equation}
We claim that  
\begin{equation}\label{poker2} 
2{\rm Re\,} v= -\rho(\zeta)-\rho(z)+ L\rho(\zeta) +\Ok(|\zeta-z|^3),
\end{equation}
where $L\rho(\zeta)$ is the Levi form in the Taylor expansion of $\rho$ at $\zeta$. 
In fact, from \eqref{poker1} we have, using the notation $\rho_j=\partial\rho/\partial\zeta_j(\zeta)$ etc
and $\eta_j=z_j-\zeta_j$, 
\begin{multline*}
2{\rm Re\,} v=-2\rho(\zeta)- 2{\rm Re\,}\sum_j \rho_j\eta_j  -{\rm Re\,}\sum_{jk}\rho_{jk}\eta_j\eta_k
+\Ok(|\eta|^3)=\\
-\rho(\zeta) +L\rho(\zeta) - \Big(\rho(\zeta)+2{\rm Re\,}\sum_j \rho_j(\zeta)\eta_j +{\rm Re\,}\sum_{jk}\rho_{jk}\eta_j\eta_k
+L\rho(\zeta)\Big) +\Ok(|\eta|^3)=\\
-\rho(z)+L\rho(\zeta) -\rho(\zeta)+\Ok(|\eta|^3).
\end{multline*}
Since $\rho$ is strictly \psh it follows from \eqref{poker2} that  
\eqref{mos2} holds, and since also \eqref{mos3} holds, 
the level sets of $|v|$ are the Koryani balls discussed above and \eqref{pontus} holds. 
%%  
 % 
%%%% 
% 
From \eqref{poker1} it is easy to find a $(1,0)$-form $q$, depending holomorphically on $z$, such that
\begin{equation}\label{murbruk}
v=\delta_{\zeta-z} q-\rho(\zeta).
\end{equation}

\smallskip
We now turn to the case when $\rho$ is just smooth and strictly plurisubharmonic.
Let $\chi$ be a smooth function on $[0,\infty)$ that is $1$ when $t<1/2$ and $0$ when $t>1$.
We claim that the series 
\begin{equation}\label{sebra}
v(\zeta,z)=-\sum_\alpha \frac{1}{\alpha!}\frac{\partial^\alpha\rho}
{\partial \zeta^\alpha}(\zeta)(z-\zeta)^\alpha \chi(c_{|\alpha|}|z-\zeta|)
\end{equation}
converges uniformly, with all its derivatives,  and therefore defines a smooth function 
in a \nbh of the diagonal in $\Omega\times\Omega$, if $c_k$ tends to infinity fast enough.
In fact, notice that
$$
|(z-\zeta)^\gamma\chi^{(\tau)}(c_k|z-\zeta|^k)\le c_k^{-|\gamma|}\sup|\chi^{(\tau)}|.
$$
If  
$
m_k=\sup_{|\alpha|\le k}| \partial^\alpha\rho/\partial \zeta^\alpha|/\alpha!
$
%and $a_\ell=\sup_{|\tau|\le \ell} |\chi(\tau)|$, then  
%$$
it is enough to choose $c_k$ so that for any fixed $\ell$, $c_k>> m_{k+\ell}$ for all large enough $k$
(where $>>$ also depends on $\ell$ derivatives of $\chi$). 
Thus the choice of $c_k$ depends on the ultra-differentiable class of $\rho$.
We also claim that $v$ is almost anti-holomorphic in $\zeta$ in the sense that
\begin{equation}\label{stork2}
\partial_\zeta v=\Ok(|\zeta-z|^\infty).
\end{equation}
To see this, given a positive integer $\nu$, let us write \eqref{sebra} as $A+B$, where $A$ is the sum over $|\alpha|\le \nu$.
Now $\partial_\zeta A$ becomes a telescoping sum plus the terms where  $\partial_\zeta$ falls on $\chi$.
The sum gives rise to terms that are $\Ok(|\zeta-z|^\nu)$, whereas the remaining terms vanish close to the diagonal.
Clearly $\partial_\zeta B= \Ok(|\zeta-z|^\nu)$. Thus \eqref{stork2} holds.

\begin{remark}
Notice that $v(\bar\zeta,z)$ is a smooth extension of $-\rho(\bar\zeta)$
from the totally real
subspace $\{\zeta=\bar z\}$ of $\{(\bar\zeta,z);\ (\zeta,z)\in \Omega_\zeta\times\Omega_z\}$
such that  $\dbar v(\bar\zeta,z)=\Ok(|z-\bar\zeta|^\infty)$. That is,   $v(\bar\zeta,z)$ is a so-called almost holomorphic
extension. Such extensions are well-known in the literature and can be constructed in many ways. %, see, e.g.,  \cite{Horm}. 
\end{remark}

Again one can find $q$ that is holomorphic in $z$ such that \eqref{murbruk} holds. 
Moreover, 
$\overline{v(z,\zeta)}-v(\zeta,z)=\Ok(|\zeta-z|^\infty)$
but this property is not used  in this paper.  

\smallskip

We extend $v$ to $\Omega\times\Omega$ by patching with $|\zeta-z|^2$, that is,
if  $\eta=\zeta-z$ we let  
$$
\tilde v= \chi(|\eta|^2)v+(1-\chi(|\eta|^2)|\eta|^2, \quad \tilde q=\chi(|\eta|^2)q+(1-\chi(|\eta|^2)\partial |\eta|^2.
$$
so that 
$
\tilde v=\delta_{\zeta-z} \tilde q-\rho(\zeta).
$
In what follows, for simplicity,  we write $v$ and $q$ even for the extensions.

\begin{remark}\label{laban} % 
Assume that we have an
embedding $\psi\colon \Omega\to \Omega'$ into a higher dimensional ball $\Omega'\subset \C^{n'}_{\zeta'}$, 
$D'\subset\subset\Omega'$ is the unit ball, and $D=\psi^{-1} D'$. Assume in addition that
$\psi(\Omega)$ intersects $\partial D'$ transversally. 
Then $\rho=|\phi|^2-1$ is a strictly pseudoconvex defining function for $D$ 
and $v(\zeta,z)=1-\bar\psi(\zeta)\cdot\psi(z)$ is globally defined in $\Omega\times\Omega$,
holomorphic in $z$, anti-holomorphic in $\zeta$, and equal to $-\rho$ in the diagonal.  Moreover, 
one can find a $(1,0)$-form $q$ in $D$,  depending holomorphically on $z$, such that
\eqref{murbruk} holds.
%$
%v=\delta_{\zeta-z}q-\rho(\zeta).
%$
 \end{remark}

\begin{ex}\label{laban2}
There are non-trivial domains that admit a $v$ as in Remark~\ref{laban}.
One can check that 
$D=\{ z\in \C^2;\  |z_1|^2 + |z_2|^2+4|1-z_1z_2|^2 <3\}$ is strictly pseudoconvex. 
It is the inverse image of the unit ball in $\C^3$ under the mapping 
$\psi(z_1,z_2)=(1/\sqrt{3})(z_1, z_2, 2(1-z_1z_2))$ and hence
there is a global $v(\zeta,z)$ in $D$ as in the remark. Notice that $\alpha=dz_1/z_1$ is a
closed $1$-form in $D$ that is not exact since its integral over the cycle $\theta\mapsto (e^{i\theta},e^{-i\theta})$
is $2\pi i\neq 0$. Thus $H^1(D,\C)\neq 0$.
 %%%
\end{ex}

\subsection{The weight $g^\alpha$}\label{galpha}
Let $\alpha$ be any complex number. We  claim that for each fixed $z\in D$, 
\begin{equation}\label{pommes}
g^\alpha=\Big(1+\nabla_{\zeta-z} \frac{q}{-\rho}\Big)^{-\alpha+1}
=\Big(\frac{v}{-\rho}+\dbar\frac{q}{\rho}\Big)^{-\alpha+1}
\end{equation}
is a weight with respect to $z$.  %for any complex $\alpha$. 
In fact, the scalar term within the second brackets has positive real part in view of \eqref{mos2} and hence
$g^\alpha$ is well-defined by elementary functional calculus, see \cite{Aint},
and  $\nabla_{\zeta-z}g=0$ since $\nabla_{\zeta-z}^2=0$.
It is also clear that $g^\alpha_{0,0}=1$ when $\zeta=z$. Thus the claim holds.
 
A simple computation gives that 
\begin{equation}\label{pater}
g^\alpha=\sum_{k=0}^N c_{k,\alpha}\frac{ (-\rho)^{\alpha} \beta_k}{v^{\alpha+k+1}},
\end{equation}
where $c_{\alpha,k}$ are constants and $\beta_k$ are $(k,k)$-forms that are smooth in $\Omega$.  
If ${\text Re\,} \alpha$ is positive, then $g^\alpha$ vanishes on the boundary of $D$ for each fixed $z\in D$.

In case $D$ is the ball, or if, e.g., as in Remark~\ref{laban},
this weight depends holomorphically on $z\in D$. In general it only depends
smoothly on $z$; however, \eqref{stork2} holds, which is crucial in the proofs of Theorems~\ref{thmA} and \ref{thmB}.

\begin{remark}
For a given $X$ in Theorem~\ref{thmA} or Theorem~\ref{thmB}  it is in fact enough for our proofs
to choose $v$ such that
$
\partial_\zeta v=\Ok(|\zeta-z|^{\nu})
$
for a large enough $\nu$.  Such a $v$ is obtained by restricting the sum \eqref{sebra} to
$|\alpha|\le \nu+1$;
%t 
then of course the factors $\chi(c_{|\alpha|}|z-\zeta|^2)$ are not needed.
\end{remark}

\section{Proof of Theorem~\ref{thmA} when $D$ is the ball}\label{ball}
Let us first assume  that our $v(\zeta,z)$ is defined in $\Omega\times\Omega$,
holomorphic in $z$ and anti-holomorphic in $\zeta$, as in the case with the ball.
Recall, cf.~\eqref{pater},  that  for fixed $z\in D$, the weight
$g^\alpha$ vanishes to order $\alpha$ at $\partial D$. If we define it as $0$ outside
$D$, it is therefore of class  $C^{\alpha-1}$.

%Notice that the component $g^\alpha_{n,n}$ of $g$ of bidegree $(n,n)$ has the form
%$$
%\frac{(-\rho)^{\alpha}}{v^{n+\alpha+1}}\beta.
%$$

\begin{lma}\label{prutt}
If $\alpha$ is large enough and $\phi$ is holomorphic in a \nbh of $X\cap \overline D$, then
\begin{equation}\label{utvid}
\Phi(z)= \int _{\zeta\in D} HR\w \frac{(-\rho)^{\alpha}\beta_n }{v^{n+\alpha+1}} \phi
\end{equation}
is a holomorphic extension of $\phi$ to $D$. 
\end{lma}

\begin{proof}
Assume that $\alpha$ is larger than the order of the currents $U$ and $R$.
Notice that the function that is $(-\rho)^{\alpha}$ in $D$ and $0$ outside $D$ is
in $C^{\alpha-1}$. For each fixed $z\in D$, therefore $g^\alpha$, defined as $0$ outside $D$, is a weight
in $\Omega$ of class $C^{\alpha-1}$.  Thus \eqref{basic2} holds with $g=g^\alpha$. 
As in Section~\ref{poker} we conclude that  \eqref{krokant}, that is, \eqref{utvid}, is a holomorphic
extension of $\phi$ to $D$. 
%for each $z\in D$ the "integral"
%has a meaning provided that $\phi$ is defined in a  \nbh of $X\cap \overline D$.
%Furthermore, since there is some holomorphic extension $\Psi$
%of $\phi$ to a pseudoconvex \nbh of $\overline D$, we conclude from  \eqref{basic2} that 
%\eqref{utvid} is a holomorphic extension of $\phi$ to $D$, cf.~\eqref{}  
\end{proof}

We shall now make an a~priori estimate of $\Phi$ in terms of $\phi$. Let us assume that
$\phi$ is defined on $D'\cap X$, where $D'\supset\supset D$. 
Let us also assume that $X$ is defined and Cohen-Macaulay in $\Omega$.
Then we can assume that our Hermitian resolution $(E,f)$ has length $\kappa=N-n$,
and hence $HR=H^0_\kappa R_\kappa$ has bidegree $(\kappa,\kappa)$.  

If either $|\zeta-z|\ge \epsilon$ or $\zeta$ is far from
$\partial D$, then $|v|$ is strictly positive in view of \eqref{pontus}. 
By a suitable partition of unity we therefore have to estimate the $L^p$-norm of a finite number of
terms 
\begin{equation}\label{term}
\int _{\zeta\in D} HR\w \frac{(-\rho)^{\alpha}}{v^{n+\alpha+1}}\beta(\zeta,z)\phi,
\end{equation}
where $\beta$ is smooth with compact support in a small \nbh $U$ of a point $p\in \partial D\cap Z$,
plus some terms with no singularity at all. 

 Let us consider a term \eqref{term}.  Let us change notation and replace $\zeta$ by 
 coordinates $(\zeta,\tau)$ in $U$ such that
$Z\cap U=\{\tau=0\}$.   Let $\mu$ be one of the components of $R=R_\kappa$.
It is a Coleff-Herrera current, cf.~Section~\ref{kodiak}, so
% (or think of everything vector-valued). 
it can be written %we may assume that
$$
\mu=\gamma(\zeta,\tau)\dbar\frac{d\tau}{\tau^{M+\1}}\w d\zeta
$$
as in \eqref{tomat3}, cf.~\eqref{hatmu}.    Let us incorporate $H$ in $\beta$.
Integrating with respect to $\eta$, that is, taking the push-forward $\pi_*$, where $\pi$ is the projection
$(\zeta,\tau)\mapsto \zeta$, we get,  see \eqref{hatmuintegral}, 
$$
\int_{\zeta\in Z\cap D} \partial_\tau^M|_{\tau=0} \Big(\frac{(-\rho)^{\alpha}}{v^{n+\alpha+1}}\beta(\zeta,\tau;z)\gamma\phi\Big).
$$
Here $\partial^M_\tau$ stands for $\partial^{|M|}/\partial \tau^M$. In what follows
it is understood that we evaluate at $\tau=0$ after applying this operator and we thus omit $|_{\tau=0}$. 
Using that $v$ is anti-holomorphic in $(\zeta,\tau)$ we have 
$$
\sum_{m\le M}\int_{\zeta\in Z\cap D}
\frac{1}{v^{n+\alpha+1}} {M\choose m} \partial^{M-m}_\tau( (-\rho)^\alpha \beta) \partial_\tau^{m}(\gamma\phi).
$$
Thus we get a sum of terms of the form
$$
\int_{\zeta\in Z\cap D}
\frac{1}{v^{n+\alpha+1}} (-\rho)^{\alpha-\ell} \beta \partial_\tau^{m}(\gamma\phi),
$$
where $\ell\le |M-m|$ and $\beta$ is smooth. 
Since $\rho$ is a defining function we may assume that $\partial\rho$ is nonzero in $U$.  If  
$$
T=\frac{1}{|\partial\rho|^2}\sum_j \frac{\partial \rho}{\partial\bar \zeta_j}\frac{\partial}{\partial\zeta_j},
$$
then $T\rho=1$ and hence
$$
(-\rho)^{\alpha-\ell} \beta=\beta' T(-\rho)^{\alpha-\ell+1}.
$$
where $\beta'=\beta/ (\ell-\alpha-1)$.  
If $T'$ is the formal adjoint of $T$, again using that $v$ is anti-holomorphic in $\zeta$, 
$(-\rho)^{\alpha-\ell+1}=0$ on $\partial D$, and $\phi$ is defined on $X\cap\overline D$, an integration by 
parts gives
$$
\int_{\zeta\in Z\cap D}
\frac{1}{v^{n+\alpha+1}} (-\rho)^{\alpha-\ell+1} T'\big(\beta' \partial_\tau^{m}(\gamma\phi)\big).
$$
Repeating this procedure $\ell$ times we see that our extension is $\Phi$ a finite sum of terms
\begin{equation}\label{taktik}
A(z)=\int_{\zeta\in Z\cap D}
\frac{1}{v^{n+\alpha+1}} (-\rho)^{\alpha} \beta \partial_\zeta^a\partial_\tau^{m}(\gamma\phi),
\end{equation}
where $a$ is a multiindex such that $|a|\le |M-m|$ and $\beta$ is smooth.  It follows from \eqref{tomat2} and \eqref{tomat4} that 
\begin{equation}\label{trump}
|\partial_\zeta^a\partial_\tau^{m}(\gamma\phi)|\lesssim |\phi|_X.
\end{equation}
Assume that $r>-1$. Provided that $\alpha$ is large enough, from \eqref{taktik}, \eqref{trump}, and 
\eqref{polo1} in Lemma~\ref{uppsk} below we have
$$
\int_{z \in D} \delta^r |A| \lesssim
\int_{\zeta\in Z\cap D}
(-\rho)^{N-n+r} |\phi|_X.
$$
Summing up all terms $A$ we get the desired a~priori estimate \eqref{trauma} in case $p=1$. 
Now assume that $1<p<\infty$. 
Let us choose $\epsilon>0$ so that $\alpha-(q/p)\epsilon>-1$.
By  H\"older's equality  and \eqref{polo3} below,  
$$
|A|^p\lesssim \Big(\int_{\zeta\in Z\cap D}\frac{\delta^{\alpha-\frac{q}{p}\epsilon}}{|v|^{n+\alpha+1}}\Big)^{p/q}
\int_{\zeta\in Z\cap D}\frac{\delta^{\alpha+\epsilon}}{|v|^{n+\alpha+1}}|\phi|^p_X\lesssim
\delta(z)^{-\epsilon}\int_{\zeta\in Z\cap D}\frac{\delta^{\alpha+\epsilon}}{|v|^{n+\alpha+1}}|\phi|^p_X.
$$
If in addition $r-\epsilon>-1$, an application of \eqref{polo1} gives 
\eqref{trauma}.

\smallskip

If $\phi$ is just defined in $X\cap D$ we apply the same construction and 
argument to the slightly smaller strictly pseudoconvex domains
$D_\epsilon=\{\rho<-\epsilon\}$.  It is not hard to see that the same computation works in $D_\epsilon$,
with estimates that are uniform in
$\epsilon$.
We thus get
$\Phi_\epsilon$ in $D_\epsilon$ that interpolate $\phi$ in $D_\epsilon\cap X$ such that
\begin{equation}\label{tomte}
\int_{D_\epsilon} \delta^{r}_\epsilon | \Phi_\epsilon|^p dV_D \le 
C_{r,p}^p \int_{Z\cap D_\epsilon} \delta_\epsilon^{\kappa+r} |\phi|^p_X dV_Z,
\end{equation}
where $C_{r,p}$ is uniform in $\epsilon$ and $\delta_\epsilon\sim-\rho_\epsilon:=-(\rho+\epsilon)$
is the distance to $\partial D_\epsilon$.  Clearly the right hand sides of \eqref{tomte} is dominated by
the right hand side of \eqref{trauma}. If this is finite, hence
there is a subsequence $\Phi_{\epsilon_j}$ converging
to a function $\Phi$ in $D$ uniformly on compact sets in $D$.  
In particular, the convergence is in $\E(D)$,  and since 
$$
(\Phi_{\epsilon_j}-\phi)\mu=0
$$
for all $\mu\in \Homs(\Ok_\Omega/\J, \CH_\Omega^Z)$ on compact subsets of $D$,
this must hold for $\Phi$ as well, cf.~Section~\ref{kodiak}. Thus  $\phi$ is the image of $\Phi$ in $\Ok_X$, that is, 
$\Phi$ is an extension of $\phi$.  Clearly $\Phi$ satisfies \eqref{trauma} and 
thus Theorem~\ref{thmA} is proved in case 
when $D$ is a ball and $\Ok_X$ is Cohen-Macaulay.

\begin{remark}
One can check that the limit 
$$
\Phi(z)= \int _{\zeta\in D} HR\w \frac{(-\rho)^{\alpha}\beta_n }{v^{n+\alpha+1}} \phi
=\lim_{\epsilon\to 0} \int _{\zeta\in D_\epsilon} HR\w \frac{(-\rho_\epsilon)^{\alpha}\beta_n }{v_\epsilon^{n+\alpha+1}} \phi
$$
exists for each $z\in D$, and  thus it is not necessary to take a subsequence in the argument above.
However, we do not need this refinement and omit the details.
\end{remark}

We will now point out how to estimate \eqref{utvid} if $\Ok_X$ has non-Cohen-Macaulay points in $D$.
Then, cf.~Section~\ref{kodiak},
$$
HR=H^0_{\kappa}R_{\kappa}+ \cdots +H^0_{N-1}R_{N-1}.
$$
Recall the representations \eqref{pjosk1}. Since $\Ok_X$ is Cohen-Macaulay at points on $\partial D\cap Z$,
$a_k$ are smooth there and hence we  can proceed
in the same way as before at such points.  %points on $Z\cap \partial D$.

Let  $U\subset\subset Z\cap D$  be a small \nbh of a point on $Z\cap D$ and let us 
choose coordinates $(\zeta,\tau)$ in $U$ as before. Then we have, cf.~\eqref{pjosk1} and \eqref{tomat3}, that
$$
H^0_k R_k= a_k \mu= a_k \gamma \dbar\frac{d\tau}{\tau^{M+\1}}.
$$ 
Since we are far from the boundary $1/v$ is  bounded and thus we get
terms like  
$$
\int_{(\zeta,\tau)\in D} R \w \beta \phi=
\int_{(\zeta,\tau)\in D}
a_k \beta \dbar\frac{d\tau}{\tau^{M+\1}}  \w \gamma  \phi,
$$
where $\beta$ is smooth  and has compact support in $U$; also $H$ is
incorporated in $\beta$ here.

Integrating with respect to  $\tau$, that is, applying $\pi_*$,
we get by Lemma~\ref{baka} a sum of terms like
\begin{equation}\label{tror}
\int_{\zeta\in Z\cap D} b_m(\cdot, z) \partial_\tau^{m}(\gamma\phi)
\end{equation}
for $m\le M$, where $b_m(\zeta,z)$ are currents with compact support in $U$ that depend holomorphically on $z$ 
 %This means that \eqref{tror} will involve derivatives of the holomorphic function $\partial_\tau^{m}(\gamma\phi)$
in $D$.   %However, 
By usual Cauchy estimates,  \eqref{tror} is controlled by 
the $L^p$-norm of $\partial_\tau^{m}(\gamma\phi)$ over $U$. 
In view of \eqref{tomat2} and \eqref{tomat3} 
we get the same a~priori estimate as before. Thus 
Theorem~\ref{thmA} is fully proved in the case when $D$ is the ball, except for the following two
lemmas.

\begin{lma} \label{baka}
With the notation in the proof, let  $a=\beta a_k$,
and $\psi=\gamma\phi$. 
Then
$$
\pi_*\Big( \dbar \frac{d\tau}{\tau^{M+\1}}  a \phi\Big)=
\sum_{m\le M} b_m\partial_\tau^{m}\psi|_{\tau=0},
 $$
where $b_m$ are currents on $U$ with compact support in $U$. 
If, in addition, $\beta$ depends holomorphically on a parameter $z$, then
also $b_m$ will do. 
\end{lma}

\begin{proof} 
Recall from Section~\ref{kodiak}, cf.~\eqref{pjosk2}, that 
$$
a\dbar(d\tau/\tau^{M+\1})=\lim_{\epsilon\to 0} \chi(|f|^2/\epsilon) d\dbar(d\tau/\tau^{M+\1}),
$$
where $f$ is a holomorphic tuple with zero set $W$. 
It follows 
that $ \tau^{M'} a \dbar(d\tau/\tau^{M+\1})=0$ if $\tau^{M'}$ is in the ideal
$\la \tau^{M+\1}\ra$, that is, if $M'_j\ge M_j+1$ for some $j$.
Since $\psi$ is holomorphic we have % 
$$
\psi(\zeta,\tau)=\sum_{m\le M} \psi_m(\zeta)\tau^m + \cdots
$$
where $\cdots$ are terms in $\la \tau^{M+\1}\ra$.  %% 
It follows that
$$
a \psi \dbar(d\tau/\tau^{M+\1})=\sum_{m\le M} \psi_m(\zeta) a  \dbar(d\tau/\tau^{M-m+\1}),
$$
and hence
$$
\pi_*(a \psi \dbar(d\tau/\tau^{M+\1}))=\sum_{m\le M} \psi_m(\zeta) \pi_*( a  \dbar(d\tau/\tau^{M-m+\1})).
$$
Now the lemma follows, since the last factor depends holomorphically on  $z$.  
\end{proof}

\begin{lma}\label{uppsk}
With the notation above we have, for $s>-1$ and $b>0$,  we have the estimates
\begin{equation}\label{polo1}
\int_{z\in D}\frac{\delta(z)^s dV(z)}{|v|^{N+1+s+b}}\lesssim \frac{1}{\delta(\zeta)^b}
\end{equation}
and 
\begin{equation}\label{polo3}
%\ 
\int_{\zeta\in Z\cap D}\frac{\delta(\zeta)^sdV(\zeta)}{|v|^{n+1+s+b}}\lesssim \frac{1}{\delta(z)^b}.
\end{equation}
\end{lma}

This lemma is well-known and follows in a standard way from the  local 
representation \eqref{pontus} of $|v|$.  For instance, \eqref{polo1}  is reduced to the elementary estimate
$$
\int_{|x|<1, x_1>0}
\frac{x_1^s dx_1\cdots dx_{2N}}{\big(x_1+y_1+|x_2-y_2|+\sum_{j=3}^{2N} |x_j-y_j|^2\big)^{N+1+s+b}}\lesssim
\frac{1}{y_1^b}.
$$
Notice that the "worst case" in \eqref{polo3} is
when $z$ lies on $Z$. Therefore, it follows from \eqref{polo1} applied to $Z\cap D$.
See, e.g., \cite[V.3.3]{Range} for a detailed discussion of this kind of estimates.

% and the fact that $|F|\ge ???$.   ??????

\begin{remark}\label{reps0}
There is a somewhat different way to construct holomorphic extensions from $X$, which is,
e.g., used in \cite{Amar}.
Let $(E,f)$ be a Hermitian
resolution of $\Ok_D\J$ as before and let $\nabla_f=f-\dbar$, cf.~\cite{Aint2,AW1}.  The associated currents
$U$ and $R$ are related by the formula $\nabla_f U^0=I-R$, that is, 
$f_{k+1}U^0_{k+1}-\dbar U^0_k=I-R_k$,  $k=0,1,\ldots$.  
If $\phi\in\Ok(X\cap D)$,  then $R\phi$ is well-defined.
 By solving a sequence of $\dbar$-equations in $D$ one can find a current 
$V=V_1+V_2+\cdots +V_N$ in $D$ such that $f_{k+1}V_{k+1}-\dbar V_k= -R_k\phi$,  $k\ge 1$. 
We claim that $\Phi= f_1 V_1$ is a holomorphic extension of $\phi$.  Since one can solve
$\dbar$ with estimates one get estimates of $\Phi$.  In case $\kappa=1$ there is just one step in this procedure
so that if $K$ is a solution operator for $\dbar$ in $D$, then
$\Phi=f_1 K(R_1\phi)$ is a holomorphic extension of $\phi$.
However, except for the case when $X$ is reduced, 
we cannot see how to obtain Theorem~\ref{thmA} or \ref{thmB} with this approach.

Let us sketch a proof of the claim: Let $\varphi$ be any holomorphic extension of $\phi$ to $D$. Then
$$
\nabla_f U^0\varphi=(I-R)\varphi=\varphi -R\phi.
$$
Furthermore, $\nabla_f V= \Phi -R\phi$.  Hence $\nabla_f (V-U^0\varphi)=\varphi-\Phi$.  By solving
another sequence of $\dbar$-equations one can find a holomorphic $w$ such that  $\varphi-\Phi=f_1 w$.
This precisely means that $\varphi-\Phi$ is in $\J$.  
\end{remark}

\section{Proof of Theorem~\ref{thmA} in the general case}\label{ball2}
As described in Section~\ref{pokemon}, in the general case we get a similar function  
$v(\zeta,z)=\delta_{\zeta-z} q-\rho(\zeta)$
but instead of being holomorphic in $z$ and anti-holomorphic in $\zeta$ we have 
$\dbar_z v=0$ close to the diagonal $\Delta$ and the property \eqref{stork2}, respectively.
Notice, cf.~\eqref{sebra}, for future reference, that we
can choose $q$ so that $\dbar_zq=0$ close to $\Delta$. 
 
 \smallskip
In this section % and in Section~\ref{ball3}, 
 we let the $\dbar$ in $\nabla_{\zeta-z}=\delta_{\zeta-z}-\dbar$ act on 
both  $z$ and $\zeta$.  Thus also anti-holomorphic differentials
with respect to $z$ will occur in 
 $g^\alpha$, cf.~\eqref{pommes}, and in 
\begin{equation}\label{tartar}
g:=(f_1(z)H^1U+HR)\w g^\alpha.
\end{equation}
However, we only have holomorphic differentials with respect to $\zeta$.
Then still $\nabla_{\zeta-z} g^{\alpha}=0$ and $\nabla_{\zeta-z} g=0$.

Let $\Omega$ be a \nbh of $\overline D$ and assume that 
$\phi$ is defined in $\Omega\cap X$. Moreover, let $\Psi$ be a holomorphic extension to $\Omega$.
Then 
$$
\Psi(z)= \int_\zeta  g_{N,N}^{0,0}\Psi, \quad z\in D, 
$$
where upper  and lower indices denote bidegree in $z$ and $\zeta$, respectively.
Hence (the $(0,0)$-component i $z$ of)
\begin{equation}\label{ponke}
\varphi(z):= \int_{\zeta\in D} HR\w g^\alpha(\zeta,z)\phi(\zeta)
\end{equation}
is a smooth function in $D$ that interpolates $\phi$ in the sense that $\varphi-\phi$ is in $\E^{0,0}\J$,
cf.~Section~\ref{poker}.

We shall now modify the kernel in \eqref{ponke} so that it produces a holomorphic extension.
To this end we
invoke a result that should be of independent interest. We formulate
and prove a somewhat more general version in Section~\ref{ball3}.

\begin{prop}\label{C}
Assume that $\widehat D\subset\subset \widetilde D$ are pseudoconvex neighborhoods of $\overline D$. 
There is a linear operator $T\colon \E^{0,1}(\widetilde D)\cap \ker\dbar \to \E^{0,0}(\widehat D)$
such that $\dbar T\xi=\xi$ in $\widehat D$ and furthermore 
$T\xi\in \E^{0,0}\J(\widehat D)$ if $\xi\in \E^{0,1}\J(\widetilde D)$.
\end{prop}

Here $\xi\in \E^{0,1}\J(\widetilde D)$ means that $\xi$ is a smooth $(0,1)$-form in $\widetilde D$ such 
that locally $\xi$ has a representation $\xi=\xi_1 \eta_1+\cdots +\xi_\nu\eta_\nu$, where $\xi_j$ are smooth
$(0,1)$-forms and $\nu_j$ are functions in $\J$.

\smallskip
Recall from Section~\ref{pokemon} that 
$\dbar_z q=0$ and $\dbar_z v=0$ in a set $W=\{|\zeta-z|<\epsilon\}$.  It follows from \eqref{mos2} that
there is a pseudoconvex \nbh $\widetilde D$ of $\overline D$ such that
$\dbar_z g^\alpha$ is smooth in $D_\zeta\times\widetilde D_z$.  
It follows that also
$\dbar_z g$ is smooth in $\widetilde D$ for $\zeta\in D$.
Since $\nabla_{\zeta-z}g=0$, the component $g_{N,N}$ of $g$ of total bidegree $(N,N)$ is $\dbar$-closed,
and hence  
\begin{equation}\label{tok}
\dbar_z g_{N,N}^{0,0}+\dbar_\zeta g^{0,1}_{N,N-1}=0
\end{equation}
in $D\times\widetilde D$.
Since $\dbar_z q=0$ in $W$,  no anti-holomorphic differentials 
with respect to $z$ can occur in $g^\alpha$, cf.~\eqref{pommes},  there,  and hence
$g^{0,1}_{N,N-1}=0$ in $W\cap D\times \widetilde D$.

\smallskip
Notice that $\dbar_z(HR\w g^\alpha)=HR\w \dbar_z g^\alpha$.
We  now define
\begin{equation}\label{adef}
\A(\zeta,z)= T  \big(H(\zeta,t)R(\zeta)\w \dbar_t  g^\alpha(\zeta,t)\big)(z), \quad \zeta\in D, \ z\in \widehat D.
\end{equation}
Then  clearly
$$
HR\w g^\alpha(\zeta,z)-\A(\zeta,z)
$$
is holomorphic in $z\in D$.  
Thus 
\begin{equation}\label{utsikt}
\Phi(z):=\int_{\zeta\in D}\big(HR\w g^\alpha(\zeta,z)-\A(\zeta,z)\big)\phi
\end{equation}
is holomorphic in $D$. We claim that $\Phi$ is indeed an extension of $\phi$.

\begin{proof}[Proof of the claim]
 As noticed above $g_{N,N-1}^{0,1}$ vanishes in $W$. Hence it is smooth in $D$ and vanishes to high order 
at the boundary. Since $\Psi$ is holomorphic thus
$$
\int_{\zeta \in D} \dbar_\zeta g_{N,N-1}^{0,1}\Psi=0
$$
by Stokes' theorem.
In view of \eqref{tok}, cf.~\eqref{tartar},  we therefore  have
\begin{equation}\label{pucko}
\int_{\zeta\in D}  HR\w \dbar_t g^\alpha \phi=-\int_{\zeta\in D}f_1(t)H^1U\w \dbar_t g^\alpha\Psi.
\end{equation}
Applying $T$ we get 
\begin{equation}\label{hubert}
\int_{\zeta\in D}\A(\zeta,z)\phi(\zeta)=
T\Big(\int_{\zeta\in D} HR\w \dbar_t g^\alpha\phi\Big)=-
T\Big(\int_{\zeta\in D} (f_1(t)H^1U\w \dbar_t g^\alpha\Psi).
\end{equation}
 In fact, the change of order of $T$ and integration
with respect to $\zeta\in D$ is legitimate since the currents $U$ and $R$, as well as $(-\rho(\zeta))^r $
go outside and what is left are forms depending on $t$ that are smooth in $\widetilde D$.
Since 
$$
\int_{\zeta\in D}f_1(t)H^1U\w \dbar_t g^\alpha\Psi %%=\dbar_t\big(f_1(t)H^1U\w  g^\alpha\Psi\big)
$$
is  in $\E^{0,1}\J(\widetilde D)$ and
$\dbar_t$-closed,
it follows from Proposition~\ref{C} that 
$$
T\Big(\int_{\zeta\in D}f_1(t)H^1U\w \dbar_t g^\alpha\Psi\Big)
$$
is in
$\E^{0,0}\J(\widehat D)$ with respect to $z$.  We conclude that 
\eqref{hubert} is  in $\E^{0,0}\J(\widehat D)$.  Thus $\Phi-\phi$ is in $\E^{0,0}\J(D)$, and
since $\Phi$ is holomorphic, therefore $\Phi-\phi$ is in $\J$, see Lemma~\ref{polo}. Thus 
the claim is proved.
\end{proof}

Now the proof of Theorem~\ref{thmA}, that is, estimating the extension $\Phi$,
is concluded in the essentially same way as for the case with the ball in Section~\ref{ball}.
Since $\A$ has no singularities at the diagonal the second term in the definition
\eqref{utsikt} of $\Phi$ offers no problems at all. 
The first term is handled as in the proof for the ball. %Theorem~\ref{thmA}.
In fact, close to a point $\partial D\cap Z$  the same
arguments as before work. Each time a holomorphic derivative falls on $v$ we get
$\Ok(|\zeta-z|^\infty)$ which cancels the singularity in view of \eqref{mos2}.  
In a \nbh of a (possibly non-Cohen-Macaulay) point in $D\cap Z$ one proceed precisely as in the
the proof of Theorem~\ref{thmA} for the ball.

\begin{lma}\label{polo}
If $\Phi$ is holomorphic and in $\E^{0,0}\J$, then it is in $\J$.
\end{lma}

More explicitly, if $\eta_1,\ldots,\eta_\nu$ generate $\J$, $\Phi=a_1\eta_1+\cdots +a_\nu\eta_\nu$
for some smooth functions $a_j$ and $\Phi$ is holomorphic,  then one can choose
holomorphic such $a_j$.   

This  lemma should be well-known and it is 
an immediate consequence of  the first part of Proposition~\ref{stekpanna}.

%This must be well-known but we include a short proof.
%Notice that the components of $f_1$ generate $\J$.
%if $\eta_1,\ldots,\eta_\nu$ by construction $f_{1j}$, $j=1,\ldots\nu$  generates $\J$. If  and
%$\Phi=a_1f_{11}+\cdots +a_\nu f_{1\nu}$

%\begin{proof}
%It is a local statement.
%Since $\Phi$ is holomorphic we can choose a local representation formula
%of the form
%$$
%\Phi(z)=\int_\zeta (f_1(z)H^1U+HR)\w g \Phi,
%$$
%where the kernel depends holomorphically on $z$.  
%The hypothesis implies that $\Phi$ annihilates $R$, and hence the conclusion follows. 
%\end{proof}

\section{The $\dbar$-equation for forms in $\E\J$}\label{ball3}
 
In this section  $\J$ is a quite arbitrary ideal sheaf in a pseudoconvex domain
$\Omega\subset \C^N$.

\begin{thm}\label{CC}
Let $\J$ be an ideal sheaf in a pseudoconvex domain $\Omega\subset\C^N$, assume that
its zero set $Z$ has codimension $\kappa>0$, and let
 $\Omega'\subset\subset\Omega$. There is a linear operator $T\colon \E^{0,1}(\Omega)\cap\Ker\dbar\to \E^{0,0}(\Omega')$,  
such that $\dbar T\xi=\xi$ in $\Omega'$ and furthermore $T\xi\in \E^{0,0}\J(\Omega')$ if
 $\xi\in \E^{0,1}\J(\Omega)$.
\end{thm}

\begin{proof}
In a possibly slightly smaller pseudoconvex domain, that we denote by $\Omega$ as well, we can choose a
Hermitian free resolution \eqref{free} of $\Ok_\Omega/\J$. Let $U$ and $R$ be the associated currents and let
$H$ be a Hefer morphism associated with \eqref{free}. Moreover, let 
$g$ be a smooth weight with respect to 
$z\in \Omega'$ with compact support in $\Omega$, cf.~Section~\ref{kodiak2}. We also assume that
$g$ depends holomorphically on $z$.
Furthermore,  let $B$ be the component of the full 
Bochner-Martinelli form, see \cite[Section~2]{Aint}, that only has holomorphic differentials with respect to $\zeta$.
It follows from \cite[Section~7.4]{Aint2}, see also \cite{AS,AL}, that if $v$ is a smooth $(0,0)$-form
in $\Omega$, then 
\begin{equation}\label{gata}
v(z)=\int_{\zeta\in\Omega} (f_1(z)H^1U+HR)\w g \w B \w\dbar v + \int_{\zeta\in\Omega} (f_1(z)H^1U+HR)\w g  v
\end{equation}
for $z\in \Omega'$.  In fact, one can choose regularizations $U^\epsilon$ and $R^\epsilon$
of $U$ and $R$, respectively, so that
$$
g^\epsilon=f_1(z)H^1U^\epsilon+HR^\epsilon
$$
are smooth weights, and then  
\begin{equation}\label{gata2}
v=\int_{\zeta\in\Omega} g^\epsilon \w g\w B\w \dbar v+ \int_{\zeta\in\Omega}  g^\epsilon \w g  v
\end{equation}
holds for $\epsilon>0$, see Remark~\ref{reps} and, e.g., \cite{AS}. 
Now 
$$ %\begin{equation}
g^\epsilon\to g':=f_1(z)H^1U+HR
$$
 as currents when $\epsilon\to 0$.
Notice  that $g'\w B$ is a tensor product of currents and hence well-defined in $\Omega\times\Omega$,
and that $g^\epsilon\w B\to g'\w B$.  Thus \eqref{gata} follows from \eqref{gata2}.

\smallskip

Let $\psi$ be a $\dbar$-closed smooth $(0,1)$-form in $\Omega$ and 
let $v$ be a (smooth) solution to $\dbar v=\psi$ in $\Omega$.  Since the second term in \eqref{gata} is holomorphic,
it follows that 
\begin{equation}\label{anka1}
T\psi:=\int_{\zeta\in\Omega} (f_1(z)H^1U+HR)\w g \w B \w \psi
\end{equation}
is a solution to $\dbar u=\psi$ in $\Omega'$. Since two solutions only differ by a holomorphic function it is 
clear that $T\psi$ is smooth. This is also seen directly, noticing that
\begin{equation}\label{anka2}
T\psi= v-\int_{\zeta\in\Omega} (f_1(z)H^1U+HR)\w g  v.
\end{equation}

Now assume that, in addition, $\psi\in\E^{0,1}\J$.  Then $R\psi=0$ and thus 
$HR\w g \w B \w \psi$ vanishes since it is a tensor product  of $R\psi$ and
$B$ times smooth forms.
Therefore, cf.~\eqref{anka1},
 $$
u:=T\psi(z)=f_1(z)\int_{\zeta\in\Omega} H^1U\w g \w B\w \psi=:   f_1(z) b(z).  
$$
However, we do not know that $b$ is smooth; in fact it is  (probably) not in general, 
and hence we cannot conclude directly that $u\in\E^{0,0}\J$.  Notice for instance that $1=f(1/f)$ although  $1$ is not in $\la f \ra$.  
To prove that $u$ is indeed in $\E^{0,0}\J$ we first use the following lemma.

\begin{lma}
If $\psi\in\E^{0,1}\J$, $\dbar\psi=0$, then $R u=0$.
\end{lma}

Since $u$ is smooth, $Ru$ is well-defined.

\begin{proof}
Let $R_z$ denote $R$ depending on $z$. 
First notice that $R_z\w U$ is a well-defined current in $\Omega_\zeta\times\Omega_z$ since it is
a tensor product. Moreover, $B$ is an almost semi-meromorphic form
and therefore, cf.~\eqref{pjosk2},  
$$
R_z\w H^1U\w B:=\lim_{\epsilon\to 0} R_z\w H^1U\w B^\epsilon
$$
is a well-defined current, where $B^\epsilon=\chi(|\zeta-z|^2/\epsilon)B$. See also \cite{AS,AL,AW3}.

Since $u$ is smooth and $R_z^\epsilon\to R_z$ we have that $R_z^\epsilon u\to R_z u$.
Moreover,
$$
R_z^\epsilon u=\int_{\zeta\in\Omega} R^\epsilon_z\w f_1(z)H^1U\w B\w g \psi.
$$
We claim that
\begin{equation}\label{anka3}
W_k=\lim_{\epsilon\to 0} R^{\epsilon}_{z,k}\w H^1U\w B-R_{z,k}\w H^1U\w B=0, \quad k=0,1,\ldots.
\end{equation}

The proof of this claim relies on the fact that all currents involved are {\it pseudomeromorphic} and that such currents
fulfill the dimension principle: If $\mu$ is pseudomeromorphic, has bidegree $(*,\ell)$, and support on
a subvariety of codimension strictly larger than $\ell$, then $\mu$ must vanish. See \cite{AW2,AS}. 

\begin{proof}[Proof of the claim]
Since $R_{z,k}\w U$ is a tensor product, $R^\epsilon_{z,k}\w U\to R_{z,k}\w U$. Since 
$B$ is smooth outside the diagonal $\Delta$, therefore $W_k=0$ there. That is,
$W_k$ has support on $\Delta$. 
 
Recall that $H^1U$ is a sum of currents of bidegree $(*,*)$ in $\zeta$
so that $H^1U\w B$ is a sum of currents of bidegree at most $(N,N-1)$. 
Thus $W_k$ has bidegree at most $(N,N-1+k)$.
Since $R_k$ has support on $Z$ we have that
$W_k$ has support on $\Delta\cap \Omega\times Z$ which we can think of as
 $Z\subset \Delta\subset \Omega\times\Omega$, and hence it has 
 codimension $N+\kappa$ in $\Omega\times\Omega$.  By the dimension principle
 we conclude that $W_k=0$ if $k\le \kappa$. 

Next we use the
fact that outside a Zariski closed set $Z_{1}\subset Z$ with codimension at least $1$ in $Z$ 
there is a smooth form $\alpha_{1}$ such that $R_{\kappa+1}=\alpha_{1}R_\kappa$,
see, \cite{AW1}.  Outside $Z_1$ 
thus $W_{\kappa+1}=\alpha_{1}W_\kappa=0$.  Thus $W_{\kappa+1}$ has anti-holomorphic degree
at most $N-1+\kappa+1$ and  support on
$Z_{1}\subset \Delta\subset \Omega\times\Omega$.  Again by the dimension principle
it must vanish.
In general, there are Zariski closed sets  $Z_\ell\subset Z$ of codimension at least $\ell$ in
$Z$, and smooth forms $\alpha_\ell$ outside $Z_\ell\subset Z$
such that $R_{\kappa+\ell+1}=\alpha_{\ell+1} R_{\kappa+\ell}$ there. 
 The claim now follows by finite induction. 
\end{proof}

From the claim we conclude that
$$
R_z T\psi(z)=\int_{\zeta\in\Omega} R_z f_1(z) H^1U\w g \w B\w \psi 
=\lim_{\epsilon\to 0} \int_{\zeta\in\Omega} R_z f_1(z) H^1U\w g \w B^\epsilon \w \psi =0,
$$
where the last equality holds since $R_z f_1(z)=0$ and hence the tensor product (times smooth forms)
$R_z f_1(z) H^1U\w B^\epsilon\w \psi$ vanishes as well.
Thus the lemma is proved.
\end{proof}
 
%\begin{remark}
%Alternatively one can notice that 
%$$
%u_\epsilon:=  v-\int_\zeta  f_1(z)H^1U^\epsilon \w g  v
%$$
%tends to $u=T\psi$ if $\psi$ is in $\E^{0,1}\J$.  One then proves that 
%$R_k u_\epsilon\to R_k$ in basically the same way as in the proof of the claim above. 
%Since $R u_\epsilon=0$ for each $\epsilon>0$ in view of ???, the lemma then follows.
%\end{remark}

We can now conclude the proof of Theorem~\ref{CC}.  
Since $\dbar u=\psi$, that is, 
$$
\partial u/\partial\bar z_j=\psi_j, \quad j=1 \ldots, N, 
$$
where each $\psi_j$ is in $\E^{0,0}\J$, we conclude that
$$
(\partial^\alpha u/\partial \bar z^\alpha u) R=0
$$
for all $\alpha\ge 0$.  It now follows from \cite[Theorem~5.1]{AW1}
that $u$ is in $\E^{0,0}\J$.  
\end{proof}

\begin{remark}\label{reps}
If $f$ is a holomorphic tuple that vanishes on $Z$ and $\chi(t)$ is as before then
one can take $U^\epsilon=\chi(|f|^2/\epsilon)U$ and then define $R^\epsilon$ so that
$\nabla_f U^{\epsilon,0}=I-R^\epsilon$. 
% For the precise meaning of this relation, see,
%\cite{AW1}.  
Notice that $R^\epsilon_k$ may be non-vanishing for all $k\ge 0$.
\end{remark}

\section{Proof of Theorem~\ref{thmB}}\label{ball4}
 
If  $\J=\la f^{M+1} \ra$, then we have the simple resolution
$$
0\to \Ok(E_1)\stackrel{f^{M+1}}{\to}\Ok(E_0)\to \Ok/\J\to 0,
$$
where $E_1$ and $E_0$ are trivial line bundles. 
% is Cohen-Macaulay.  
%If we choose local coordinates $(\zeta,\eta)$ such that  
%$f=\eta$, then $\J$ is generated by all monomials $\eta^\alpha$ for $|\alpha|\le m+1$. 
%By Hilbert's syzygy theorem there is free resolution, just depending on $\eta$, of length $N-n$,
%at a fixed point $\zeta$, 
%and it gives also a resolution of the same length locally in $\Omega$.  
Moreover,
$$
U=\frac{1}{f^{M+1}},   
\quad R=R_1=\dbar\frac{1}{f^{M+1}},
$$
and if $h$ is a 
holomorphic $(1,0)$-form in $\Omega$ for each $z\in\Omega$ such that $\delta_{\zeta-z} h=f-f(z)$, 
then  
$$
H=\sum_{k=0}^M f(\zeta)^{M-k} f(z)^k h
$$
is a Hefer form for $f^{M+1}$, that is,
$$
\delta_{\zeta-z} H=f(\zeta)^{M+1}-f(z)^{M+1}.
$$
Thus
$$
HR=H\dbar\frac{1}{f^{M+1}}=\sum_{k=0}^M f^k(z) h\w \dbar\frac{1}{f^{k+1}}.
$$
Let us first assume  that we are in the ball so that $v(\zeta,z)$ is holomorphic in $z$ and anti-holomorphic
in $\zeta$.  Then we get our extension 
$$
\Phi(z)= \int_{\zeta \in D\cap X} \sum_{k=0}^M f^k(z)\dbar\frac{1}{f^{k+1}} \w h\w g^\alpha \phi
$$
for a suitably large $\alpha$.  Arguing precisely as in Section~\ref{ball}, cf.~\eqref{taktik},  we see that
$$
\Phi(z)= \int_{\zeta\in D\cap Z} \sum_{k=0}^M f^k(z) \frac{(-\rho)^\alpha}{v^{\alpha+n+1}} \beta_k \sum_{|\beta|=k}\partial^\beta \phi,
$$
where $\beta_k$ are smooth forms. If $\zeta\in Z$, then $f(z)=f(z)-f(\zeta)=\Ok(|\zeta-z|)$ and
hence $|f(z)|\le \sqrt{|v|}$. Using the same estimates as in Section~\ref{ball} now Theorem~\ref{thmB}
follows in the case with the ball. Combining with the arguments in Section~\ref{ball2} the general case follows.

\begin{remark}
It is reasonable to believe that it is possible to get a similar  sharpening of Theorem~\ref{thmA},
for instance, if $Z$ has higher codimension and  $\J$ is a jet ideal $\J_Z^{M+1}$.
 %we do not pursue this question here. 
\end{remark}

\end{document}